\documentclass{article}
\usepackage{graphicx}
\usepackage{amsmath,amssymb,latexsym}
\usepackage{amsthm}
\usepackage{dsfont}
\usepackage{hyperref}
\usepackage{a4wide}

\newcommand{\set}[1]{\left\{ #1\right\}}
\newcommand{\bR}{\displaystyle \mathbb{R}}
\newcommand{\bE}{\displaystyle \mathbb{E}}
\newcommand{\bP}{\displaystyle \mathbb{P}}
\newcommand{\bN}{\displaystyle \mathbb{N}}

\newtheorem{theorem}{Theorem}[section]
\newtheorem{lemma}[theorem]{Lemma}
\newtheorem{prop}[theorem]{Proposition}
\newtheorem{remark}[theorem]{Remark}

\newtheorem{defi}[theorem]{Definition}

\begin{document}
\title{Remarks on the Central Limit Theorem for Non-Convex Bodies}
\author{Uri Grupel\footnote{Supported by a grant from the European Research Council.}}
\maketitle

\begin{abstract}
	In this note, we study possible extensions of the Central Limit Theorem for non-convex bodies.
	First, we prove a Berry-Esseen type theorem for a certain class of unconditional bodies that are not necessarily convex.
	Then, we consider a widely-known class of non-convex bodies, the so-called p-convex bodies, and construct a counter-example for this class.
\end{abstract}

\section{Introduction}

	Let $X_1,...,X_n$ be random variables with
	$\bE X_i=0$ and $\bE X_iX_j=\delta_{i,j}$ for $i,j=1,2,...,n$. Let $\theta\in S^{n-1}$, where $S^{n-1}\subseteq \bR^n$ is the unit sphere centered at $0$, and let $G$ be a standard Gaussian random variable,
	that is $G$ has density function $\frac{1}{\sqrt{2\pi}}e^{-x^2/2}$. We denote $X=(X_1,...,X_n)$.
	In this paper we examine different conditions on $X$ under which $X\cdot\theta$ is close to $G$ in distribution.
	The classical central limit theorem states that if $X_1,...,X_n$ are independent then for most $\theta\in S^{n-1}$ the marginal $X\cdot\theta$ is close to $G$.
	It was conjectured by Anttila, Ball and Perissinaki \cite{ABP} and by Brehm and Voigt \cite{BV} that if
	$X$ is distributed uniformly in a convex body $K\subseteq\bR^n$, then for most $\theta\in S^{n-1}$ the marginal $X\cdot\theta$ is close to $G$.
	This is known as the central limit theorem for convex sets and was first proved by Klartag \cite{K}.\\
	
	In this note we examine extensions of the above theorem to non-convex settings.
	Our study was motivated by the following observation on the unit balls of $l_p$ spaces for $0<p<1$:\\
	
	We denote by $B_p^n=\set{x\in\bR^n;\; |x_1|^p+\cdots+|x_n|^p\leq 1}$ the unit ball of the space $l_p^n$.
	For $X=(X_1,...,X_n)$ that is distributed uniformly on $c_{p,n}B_p^n$, $p>0$, $\theta\in S^{n-1}$, and $G$ a standard Gaussian, one can show that
	$$\left|\bP(\theta\cdot X\leq t)-\bP(G \leq t)\right|\leq C_p \sum_{k=1}^n |\theta_k|^3$$
	where $c_{p,n}$ is chosen such that $\bE X_i=0$ and $\bE X_iX_j=\delta_{i,j}$ for $i,j=1,2,...,n$, and $C_p>0$ does not depend on $n$.\\

	In order to formulate our results we use the following definitions: Let $X=(X_1,...,X_n)$ be a random vector in $\bR^n$.
	A random vector $X$ is called isotropic if $\bE X_i=0$ and $\bE X_iX_j=\delta_{i,j}$ for $i,j=1,2,...,n$.
	A random vector $X$ is called unconditional if the distribution of $(\varepsilon_1X_1,...,\varepsilon_nX_n)$ is the same as the distribution of $X$ for any $\varepsilon_i=\pm 1,\; i=1,...,n$ .\\
	
	The first class of densities we define is based on Klartag's recent work \cite{K2} and includes the uniform distribution over $B_p^n$ for $0<p<1$.
	\begin{theorem}
		\label{thm:unconditional}
		Let $X$ be an unconditional, isotropic random vector with density $e^{-u(x)}$, where the function $u\left(x_1^{\kappa},...,x_n^{\kappa}\right)$ is convex in $\bR^n_+=\left\{x\in\bR^n;\; x_i\geq 0 \;\forall i\in\{1,2,...,n\}\right\}$ for $\kappa>1$. Let $G$ be a standard Gaussian random variable and $\theta\in S^{n-1}$. Then
		$$\left|\bP\left(\theta\cdot X\geq t\right)-\bP(G\geq t)\right|\leq C_{\kappa}\sum_{k=1}^{n}|\theta_k|^3,$$
		where $C_{\kappa}>0$ depends on $\kappa$ only, and does not depend on $n$.
	\end{theorem}
	
	In order to see that Theorem \ref{thm:unconditional} includes the uniform distribution over $B_p^n$ for $0<p<1$ take
	$$u(x)=\left\{\begin{matrix}
		0,& x_1^p+\cdots+x_n^p\leq 1\\
		\infty, & \text{otherwise}
	\end{matrix}\right. ,$$
	and set $\kappa=\frac{1}{p}$.\\
	
	The error rate in Theorem \ref{thm:unconditional} is the same as in the classical Central Limit Theorem. For example, by choosing $\theta=\left(\frac{1}{\sqrt{n}},...,\frac{1}{\sqrt{n}}\right)$, we get an error rate of $O\left(\frac{1}{\sqrt{n}}\right)$.\\
	
	The symmetry conditions in Theorem \ref{thm:unconditional} are highly restrictive. Hence, we are led to study p-convex bodies, which satisfy fewer symmetry conditions and are shown to share some of the properties of convex bodies.\\
	
	We say that $K\subset\bR^n$ is {\it p-convex} with $0<p<1$ if $K=-K$ and for all $x,y\in K$ and $0<\lambda<1$, we have
	$$\lambda^{\frac{1}{p}}x+(1-\lambda)^{\frac{1}{p}}y\in K,$$
	These bodies are related to unit balls of $p-norms$ and were studied in relation to local theory of Banach spaces by Gordon and  Lewis \cite{GL}, Gordon and Kalton \cite{GK}, Litvak, Milman and Tomczak-Jaegermann \cite{LMT} and others (see \cite{BBP}, \cite{D}, \cite{Ka}, \cite{L}, \cite{M}).\\
	
	The following discussion explains why the class of p-convex bodies does not give the desired result.
	\begin{theorem}
		\label{thm:weakconvex}
		Set $N=n+n^{\frac{5}{2}}\log^2 n$. There exists a random vector $X$ distributed uniformly in a $\frac{1}{2}-$convex body $K\subseteq \bR^{N}$, and a subspace $E$ with $\textrm{dim}(E)=n$,
		such that for any $\theta\in S^{N-1}\cap E$, the random variable $\theta\cdot \textrm{Proj}_EX$ is not close to a Gaussian random variable in any reasonable sense (Kolmogorov distance, Wasserstein distance and others).
	\end{theorem}
	A similar construction can be made for any fixed parameter $0<p<1$.
	Since $\textrm{dim}(E)$ tends to infinity with $n$, a similar theorem is not true in the convex case.
	Hence, the central limit theorem for convex sets cannot be extended for the p-convex case.
	Thus, we need to look for a new class of bodies (densities) that includes the $l_p^n$ unit balls, with a weaker condition than the unconditional one.\\
	
	\begin{remark}
		In \cite{L} Litvak constructed an example of a $p$-convex body for which the volume distribution is very different from the convex case.
		Litvak's work studies the large deviations regime for $p$-convex distributions, while our work is focused on the central limit theorem.
	\end{remark}
	
	Throughout the text the letters $c,C,c',C'$ will denote universal positive constants that do not depend on the dimension $n$.
	The value of the constant may change from one instance to another.
	We use $C_{\alpha},C(\alpha)$ for constants that depend on a parameter $\alpha$ and nothing else.
	$\sigma_{n-1}$ will denote the Haar probability measure on $S^{n-1}$.
	$f(n)=O(g(n))$ is the big O notation, i.e. there exists a constant $C>0$ such that $|f(n)|\leq C g(n),\; \forall n\in \bN$.\\
	
	{\it Acknowledgement.} This paper is part of the authors M.Sc thesis written under the supervision of Professor Bo'az Klartag whose guidance, support and patience were invaluable.
	In addition, I would like to thank Andrei Iacob for his helpful editorial comments. Supported by the European Research Council (ERC).
	
\section{A Class of Densities with Symmetries}
	In this section we use Klartag's recent work \cite{K2} in order to exhibit a family of functions, which includes the indicator functions of $l_p^{n}$ unit balls, for $0<p<1$, having almost Gaussian marginals.\\
	A special case of Theorem 1.1 in \cite{K2} gives us the following Lemma.
	\begin{lemma}\label{poincare} Let $\kappa> 1$ and let $\phi:\bR^n\rightarrow\bR$ be an unconditional probability density function such that $\phi(x_1^{\kappa},...,x_n^{\kappa})$ is convex on $\bR^n_+$.
	Let $X$ be a random vector with density $e^{-\phi(x)}$. Then
	$$\textrm{Var}|X|^2\leq c_{\kappa}\sum_{j=1}^n\bE|X_j|^4,$$
	where $c_{\kappa}$ depends only on $\kappa$.\end{lemma}
	
	\begin{lemma} \label{borell_type} Let $\kappa\geq 1$ and let $\phi:\bR^n\rightarrow\bR$ be an unconditional probability density function such that $\phi(x_1^{\kappa},...,x_n^{\kappa})$ is convex
	on $\bR^n_+$. Let $X$ be a random vector with density $e^{-\phi(x)}$. Then for any $p\geq 1$ and $i=1,...,n$,
	$$\bE |X_i|^p \leq c_{p,\kappa} \left(\bE |X_i|^2\right)^{\frac{p}{2}}.$$\end{lemma}
	\begin{proof}
	If $p\leq 2$ then, by H\"older's inequality, we have $c_{p,\kappa}=1$. Assume that $p\geq 2$.
	Define $\pi:\bR^{n}_+\rightarrow\bR^n_+$ by $\pi(x)=\left(|x_1|^{\kappa},...,|x_n|^{\kappa}\right)$.
	The Jacobian of $\pi$ is $\prod_{j=1}^n \kappa|x_j|^{\kappa-1}$. Using the symmetry of $\phi$ we obtain
	$$\int_{\bR^n}|x_i|^pe^{-\phi(x)}dx=2^n\int_{\bR^n_+}|x_i|^pe^{-\phi(x)}dx=2^n\int_{\bR^n_+}|x_i|^{p\kappa}\left(\prod_{j=1}^n \kappa|x_j|^{\kappa-1}\right)e^{-\phi(\pi(x))}dx$$
	Now set $u(x)=\phi(\pi(x))-(\kappa-1)\displaystyle\sum_{j=1}^n \log|x_j|$. The function $e^{-u(x)}$ is log-concave on $\bR^n_+$, with $\kappa^n\displaystyle\int_{\bR^n_+}e^{-u(x)}=\frac{1}{2^n}$, and
	$$\int_{\bR^n_+}|x_i|^pe^{-\phi(x)}dx=\kappa^n\int_{\bR^n_+}|x_i|^{p\kappa}e^{-u(x)}dx.$$
	By Borell's Lemma (see \cite{Be}, \cite{Bo}, \cite{MP}) we obtain
	$$(2\kappa)^n\int_{\bR^n_+}|x_i|^{p\kappa}e^{-u(x)}dx\leq C_{\kappa,p} \left((2\kappa)^n\int_{\bR^n_+}|x_i|^{2\kappa}e^{-u(x)}dx\right)^{\frac{p}{2}}=
	C_{\kappa,p} \left(\int_{\bR^n}|x_i|^2e^{-\phi(x)}dx\right)^{\frac{p}{2}}$$
	\end{proof}
		
	\begin{lemma} \label{thin_2} Let $\kappa> 1$ and let $\phi:\bR^n\rightarrow\bR$ be an unconditional, isotropic probability density function such that $\phi(x_1^{\kappa},...,x_n^{\kappa})$ is convex
	on $\bR^n_+$. Let $X$ be a random vector with density $e^{-\phi(x)}$. Then, for any $a\in\bR^n$

	$$\textrm{Var}(a_1^2X_1^2+\cdots+a_n^2X_n^2)\leq C_{\kappa}\sum_{j=1}^n |a_j|^4.$$\end{lemma}
	\begin{proof}
	By applying a linear transformation, Lemma \ref{poincare} gives
	$$\textrm{Var}(a_1^2X_1^2+\cdots+a_n^2X_n^2)\leq C_{\kappa}'\sum_{j=1}^n\bE a_j^4|X_j|^4.$$
	By Lemma \ref{borell_type}, we obtain
	$$\textrm{Var}(a_1^2X_1^2+\cdots+a_n^2X_n^2)\leq C_{\kappa}'\sum_{j=1}^n\bE a_j^4|X_j|^4\leq C_{\kappa}\sum_{j=1}^n a_j^4\left(\bE|X_j|^2\right)^2=C_{\kappa}\sum_{j=1}^n |a_j|^4.$$
	\end{proof}
	
	We are now ready to prove Theorem \ref{thm:unconditional}.
	\begin{proof} Since $X$ is unconditional,
	$$\bP\left(\theta\cdot X\geq t\right)=\bP\left(\sum_{k=1}^n\theta_k X_k\varepsilon_k\geq t\right),$$
	where $\varepsilon_1,...,\varepsilon_n$ are i.i.d. random variables distributed uniformly on $\{\pm 1\}$ that are independent of $X$.
	By the triangle inequality,
	$$\left|\bP\left(\sum_{k=1}^n\theta_k X_k\varepsilon_k\geq t\right)-\bP(G\geq t)\right| \leq
	\bE_X\left|\bP(G\geq t)- \bP_G\left(G\geq \frac{t}{\sqrt{\sum_{k=1}^n\theta_k^2 X_k^2}}\right)\right|$$
	$$+\bE_{X}\left| \bP_{\varepsilon}\left(\sum_{k=1}^n \varepsilon_k\theta_kX_k\geq t\right)-\bP_G\left(G \geq \frac{t}{\sqrt{\sum_{k=1}^n\theta_k^2X_k^2}}\right)\right|.$$
	We estimate each term separately. Denote $Y_n=\sum_{k=1}^n \theta_k^2X_k^2$. By the Berry-Esseen Theorem (see \cite{F}),
	\begin{align*}&\bE_{X}\left| \bP_{\varepsilon}\left(\sum_{k=1}^n \varepsilon_k\theta_kX_k\geq t\right)-\bP_G\left(G \geq \frac{t}{\sqrt{Y_n}}\right)\right|&\\
	&\leq C\left(\bE_X \sum_{k=1}^n \frac{|\theta_k|^3|X_k|^3}{\left(Y_n\right)^{\frac{3}{2}}}1_{\left[\frac{1}{2},\infty\right)}\left(Y_n\right)+2\bP\left(Y_n<\frac{1}{2}\right)\right)&\\
	&\leq C\left(10\sum_{k=1}^n \bE_X |\theta_k|^3|X_k|^3+2\bP\left(Y_n<\frac{1}{2}\right)\right)\leq C_{\kappa}\sum_{k=1}^n |\theta_k|^3 +C\bP\left(Y_n<\frac{1}{2}\right)&\end{align*}
	Here we used Lemma \ref{borell_type} to estimate $\bE|X_k|^3 $. Note that
	$$\bE_XY_n=\bE_X \sum_{j=1}^n\theta_j^2X_j^2=\sum_{j=1}^n\theta_j^2\bE_XX_j^2=\sum_{j=1}^n\theta_j^2=1,$$
	so by Chebyshev's inequality and Lemma \ref{thin_2}
	\begin{equation}
		\bP\left(\left|Y_n-1\right|\geq \frac{1}{2}\right)\leq \frac{Var\left(Y_n\right)}{\frac{1}{4}}\leq
	4C_{\kappa}\sum_{j=1}^n|\theta_j|^4
		\label{eq:var}
	\end{equation}
	Hence, since $|\theta_i|\leq 1$ for all $i=1,...,n$,
	$$\bE_{X}\left| \bP_{\varepsilon}\left(\sum_{k=1}^n \varepsilon_k\theta_kX_k\geq t\right)-\bP\left(G \geq \frac{t}{\sqrt{Y_n}}\right)\right|\leq
	C_{\kappa} \sum_{k=1}^n |\theta_k|^3.$$
	Now, in order to estimate $\bE_X\left|\bP(G\geq t)- \bP\left(G\geq \frac{t}{\sqrt{Y_n}}\right)\right|$
	we use (\ref{eq:var}) and Klartag's argument in \cite{K4} (Section 6, Lemma 7) and conclude that it is enough to show that
	$$\bE \left(\left(Y_n-1 \right)^2\left|Y_n\geq \frac{1}{2} \right.\right)\leq C\left(\sum_{j=1}^n|\theta_j|^3\right)$$
	By Lemma \ref{thin_2} we get
	$$\bE \left(Y_n-1 \right)^2=\textrm{Var}\left(Y_n\right)\leq C_{\kappa}\sum_{j=1}^n |\theta_j|^4$$
	Hence,
	\begin{align*}\bE \left(\left(Y_n-1 \right)^2\left|Y_n\geq \frac{1}{2} \right.\right)\leq
	\bE \left(Y_n-1 \right)^2 \bP\left(Y_n\geq \frac{1}{2}\right)^{-1}\leq C_{\kappa}\left(\sum_{j=1}^n |\theta_j|^4\right)\bP\left(Y_n\geq \frac{1}{2}\right)^{-1}\end{align*}
	From inequality (\ref{eq:var}) it follows that
	$$\left(\bP\left(Y_n\geq \frac{1}{2}\right)\right)^{-1}=\bP\left(\sum_{j=1}^n\theta_j^2X_j^2\geq \frac{1}{2}\right)^{-1}\leq	\frac{1}{1-C_{\kappa}\sum_{j=1}^n|\theta_j|^4}.$$
	We may assume that $\displaystyle\sum_{j=1}^n|\theta_j|^4$ is bounded by some small positive constant depending on $\kappa$, since otherwise the result is trivial, and obtain
	$$\frac{1}{1-C_{\kappa}\sum_{j=1}^n|\theta_j|^4} \leq 1+C_{\kappa}\sum_{j=1}^n|\theta_j|^4$$
	which completes our proof.\end{proof}
	
\section{The p-Convex Case}
	In this section we construct a random vector $X$, distributed uniformly in a $\frac{1}{2}-$convex body $K$, such that for a large subspace $E\subseteq \bR^n$
	the random vector $\textrm{Proj}_E X$ has no single approximately Gaussian marginal. We define a function $f:\bR_+\rightarrow\bR_+$ such that the radial density $r^{n-1}e^{-f(r)}$ is spread across an interval of length proportional to $\sqrt{n}$; that is, we want
	$r^{n-1}e^{-f(r)}$ to be constant (or close to constant) on such an interval. Such densities have marginals that are far from Gaussian. We use the density function introduced above and an approximation argument to construct the desired body $K$.\\
	
	In order to construct a p-convex body from a function $f$, we restrict ourselves to {\it p-convex functions}.
	
	\begin{defi}
		A function $f:\bR^n\rightarrow\bR\cup\{\infty\}$ is called {\it p-convex} if for any $x,y\in \bR^n$ and $t\in[0,1]$,
	\begin{equation}
		f\left(t^{\frac{1}{p}}x+(1-t)^{\frac{1}{p}}y\right)\leq t f(x)+ (1-t)f(y).
	\label{eq:weak}
	\end{equation}
	\end{defi}
	
	The following proposition allows us to construct a p-convex body with $0<p<1$ from a p-convex function.
	
	\begin{prop} \label{func_bodies} For $\psi:\bR^n\rightarrow\bR_+$ p-convex function with $0<p<1$ and fixed $N>0$, define $f_N(x)=\left(1-\dfrac{\psi(x)}{N}\right)^N_+$.
	Then the set
  $$K_N(\psi)=\set{(x,y);\; x\in\bR^n,\; y\in\bR^N,\; |y|< f_N^{\frac{1}{N}}(x)}$$
  is p-convex.
  \end{prop}

  \begin{proof}
  	Let $(x_1,y_1),(x_2,y_2)\in K_N(\psi)$. Since $(x_i,y_i)\in K_N(\psi)$ we have
  	$f_N(x_i)>0$. Therefore,
  	$$f_N^{\frac{1}{N}}(x_i)=1-\frac{\psi(x_i)}{N}.$$
  	Let $0\leq t \leq 1$ we get
	  $$f_N^{\frac{1}{N}}(t^{\frac{1}{p}}x_1+(1-t)^{\frac{1}{p}}x_2)\geq 1-\frac{1}{N}\psi(t^{\frac{1}{p}}x_1+(1-t)^{\frac{1}{p}}x_2)\geq 1-\frac{1}{N}(t\psi(x_1)+(1-t)\psi(x_2))=$$
	  $$=tf_N^{\frac{1}{N}}(x_1)+(1-t)tf_N^{\frac{1}{N}}(x_2)> t|y_1|+(1-t)|y_2|\geq |t^{\frac{1}{p}}y_1|+|(1-t)^{\frac{1}{p}}y_2|\geq |t^{\frac{1}{p}}y_1+(1-t)^{\frac{1}{p}}y_2|.$$
	  Hence, $t^{\frac{1}{p}}(x_1,y_1)+(1-t)^{\frac{1}{p}}(x_2,y_2)\in K_N(\psi)$, as needed.
	\end{proof}
	
	\begin{prop} \label{counter_log} There exists a universal constant $C>0$ such that, for $a\geq C$ the function
	$$f(x)=\left\{\begin{matrix}
		\log a, & if\; 0\leq x\leq a\\
		\log x, & if\; a\leq x\leq 2a\\
		\sqrt{x}-\sqrt{2a}+\log 2a, & if\; 2a \leq x
	\end{matrix}\right. $$	
	is $\frac{1}{2}-$convex.\end{prop}
	
	\begin{proof}
		We begin by verifying that the function $f$ is $\frac{1}{2}-$convex for each interval $[0,a],\;[a,2a],\;[2a,\infty)$.
		Then we need to check that condition (\ref{eq:weak}) holds when $x$ and $y$ are from different intervals. By symmetry, we may assume that $x<y$.
		The cases $x,y\in[0,a]$ and $x,y\in[2a,\infty)$ are straightforward. In order for condition (\ref{eq:weak}) to hold for the function $\log x$ on an interval $[a,b]$ we must show that for any $x,y\in [a,b]$
		\begin{equation}\log((1-t)^2x+t^2y)\leq (1-t)\log(x)+t\log(y)=\log\left(x^{1-t}y^t\right). \label{weaklog} \end{equation}
		This is equivalent to
		$$(1-t)^2x+t^2y-x^{1-t}y^{t}\leq 0.$$
		Setting here $y=cx$, we obtain
		$$(1-t)^2+t^2c-c^{t}\leq 0.$$
		This inequality holds for every $1\leq c\leq 4$ and $0\leq t \leq 1$.
		To see that note that $g(t,c)=(1-t)^2+t^2c-c^{t}$ is a convex function in $c$ (as a sum of convex functions).
		Hence, it is enough to verify that $g(t,1)\leq 0$ and $g(t,4)\leq 0$ for any $0\leq t \leq 1$. Indeed,
		$$g(t,1)=(1-t)^2+t^2-1=2t(t-1)\leq 0$$
		and
		$$g(t,4)=(1-t)^2+t^24-4^{t}\Rightarrow \frac{\partial^2 g(t,4)}{\partial t^2}=2+8-(\log4)^24^t\geq 2.$$
		Hence, $g(t,4)$ is convex in $t$. Since $g(0,4)=g(1,4)=0$, we obtain, $g(t,4)\leq 0$ for all $0\leq t\leq 1$.\\
		Consequently (\ref{weaklog}) holds for any interval of the form $[a,b]\subseteq[a,4a]$.\\
		Next, we verify condition (\ref{eq:weak}) for $f$ when $x\in[a,2a]$, $y\in[2a,\infty)$, and $t^2x+(1-t)^2y\in[a,2a]$.
		We consider two cases
		\begin{enumerate}
			\item
			$y\in[2a,4a]$. By inequality (\ref{weaklog}),
			\begin{align*}
				f(t^2x+(1-t)^2y)&=\log(t^2x+(1-t)^2y)\leq t\log(x)+(1-t)\log(y)\\
				&\leq\log(x)+(1-t)(\log(2a)+\sqrt{y}-\sqrt{2a})=tf(x)+(1-t)f(y).
			\end{align*}
			The second inequality holds thanks to the elementary inequality $\log(y)-\log(2a)\leq \sqrt{y}-\sqrt{2a}$.
			Since for $y=2a$ we have equality, and $(\sqrt{y})'=\frac{1}{2\sqrt{y}}\geq \frac{1}{y}=(\log(y))'$ for $y\geq 4$, the inequality holds if $2a\geq 4$.
			
			\item
			$y\geq 4a$. Define
			$$g(t)=\log(t^2x+(1-t)^2y)-t\log(x)-(1-t)(\sqrt{y}- \sqrt{2a} +\log(2a)).$$
			We need to show that $g(t)\leq 0$ for all $t\in[0,1]$. Since $g(1)=0$, it is enough to show that $g'(t)\geq 0$ for all $0\leq t\leq 1$. We have,
			\begin{align*}g'(t)&=\frac{2tx-2(1-t)y}{t^2x+(1-t)^2y}-\log(x)+\sqrt{y}-\sqrt{2a}+\log(2a)&\\
			&\geq \frac{2tx-2(1-t)y}{t^2x+(1-t)^2y}+\left(1-\frac{1}{\sqrt{2}}\right)\sqrt{y}&\end{align*}
			Hence, if $2tx-2(1-t)y+\left(1-\frac{1}{\sqrt{2}}\right)\sqrt{y}(t^2x+(1-t)^2y)\geq 0$, then $g'(t)\geq 0$.
			Recalling that $t^2x+(1-t)^2y\geq a$, it suffices to prove that
			$$2tx-2(1-t)y+\left(1-\frac{1}{\sqrt{2}}\right)\sqrt{y}a\geq 0.$$
			Using the fact that $(1-t)^2y\leq t^2x+(1-t)^2y\leq 2a$, we obtain $(1-t)\sqrt{y}\leq \sqrt{2a}$. Hence,
			\begin{align*}
				2tx-2(1-t)y+\left(1-\frac{1}{\sqrt{2}}\right)a\sqrt{y}&\geq 2ta-2\sqrt{2a}\sqrt{y}+\left(1-\frac{1}{\sqrt{2}}\right)a\sqrt{y}\\
				&\geq\sqrt{y}\left(\left(1-\frac{1}{\sqrt{2}}\right)a-2\sqrt{2a}\right).
			\end{align*}
			This gives the condition
			$$\left(1-\frac{1}{\sqrt{2}}\right)a-2\sqrt{2a}\geq 0,$$
			Which is satisfied for $a\geq 100$.
		\end{enumerate}
		
		When $x\in[a,2a]$ and $y\in[2a,\infty)$ and $t^2x+(1-t)^2y\geq 2a$, we have
		$$f(t^2x+(1-t)^2y)=\sqrt{t^2x+(1-t)^2y}-\sqrt{2a}+\log 2a\leq t\sqrt{x}+(1-t)\sqrt{y}-\sqrt{2a}+\log 2a$$
		and
		$$tf(x)+(1-t)f(y)=t\log x + (1-t)(\sqrt{y}-\sqrt{2a}+\log 2a).$$
		Hence, (\ref{eq:weak}) holds thanks to the elementary inequality $\log 2a - \log x + \sqrt{x} - \sqrt{2a}\leq 0$, which holds for $a\geq 4$.\\
		
		If $x\in[0,a]$, then $f(x)=f(a)$ and $f(t^2x+(1-t)^2y)\leq f(t^2a+(1-t)^2y)$. Hence, for $x\in[0,a]$ and $y\in[a,\infty)$ we have
		$$f(t^2x+(1-t)^2y)\leq f(t^2a+(1-t)^2y) \leq tf(a)+(1-t)f(y) = tf(x)+(1-t)f(y).$$
	\end{proof}
	
	\begin{prop}
		Let $f:\bR_+\rightarrow\bR_+$ be a p-convex function with parameter $0<p<1$. Then $x\mapsto f(|x|)$ is a p-convex function on $\bR^n$.
	\end{prop}
	
	\begin{proof}
		First, we prove that $f$ is non-decreasing. Let $0<x<y$. There exists some $k\geq 1$ such that $2^{-k\left(\frac{1}{p}-1\right)}y\leq x$.
		We proceed by induction on $k$. For $k=1$, note that $h(t)=t^{\frac{1}{p}}y+(1-t)^{\frac{1}{p}}y$ is continuous, $h(0)=y$, and $h\left(\frac{1}{2}\right)= 2^{-\left(\frac{1}{p}-1\right)}y$.
		Hence, there exists some $0\leq t_0\leq 1$ for which $h(t_0)=x$, and so
		$$f(x)=f(t^{\frac{1}{p}}_0y+(1-t_0)^{\frac{1}{p}}y)\leq t_0f(y)+(1-t_0)f(y)=f(y)$$
		For $k\geq 2$, $f(2^{-(k-1)\left(\frac{1}{p}-1\right)}y)\leq f(y)$ by the induction hypothesis, and by the same argument as above
		$$f(x)\leq f(2^{-(k-1)\left(\frac{1}{p}-1\right)}y)\leq f(y).$$
		We thus showed that $f$ is monotone non-decreasing.
		Now, by the triangle inequality, for any $x,y\in\bR^n$ and $0<t<1$ we have
		$$f(|t^{\frac{1}{p}}x+(1-t)^{\frac{1}{p}}y|)\leq f(t^{\frac{1}{p}}|x|+(1-t)^{\frac{1}{p}}|y|)\leq tf(|x|)+(1-t)f(|y|)$$
	\end{proof}

	Using the function from Proposition \ref{counter_log}, we are ready to construct the $\frac{1}{2}-$convex body $K$ and prove Theorem \ref{thm:weakconvex}.
	
	\begin{defi}
		A sequence of probability measures $\{\mu_n\}$ on $\bR^n$ is called essentially isotropic if $\int xd\mu_n(x)=0$ and $\int x_ix_jd\mu_n(x)=(1+\varepsilon_n)\delta_{ij}$ for all $i,j=1,...,n$,
		when $\varepsilon_n	\underset{n\rightarrow\infty}{\longrightarrow} 0$.
	\end{defi}
	
	\begin{prop} \label{iso0} The probability measure $d\mu=C_ne^{-(n-1)f(|x|)}dx$ , where $f$ is defined as in Proposition \ref{counter_log}, with $a=\sqrt{\frac{3}{7}n}$, is essentially isotropic. That is,
	$$\int x_ix_jd\mu(x)=(1+\varepsilon_n)\delta_{ij}$$
	for all $i,j=1,2,...,n$, when $|\varepsilon_n|\leq \frac{C}{n}$.\end{prop}
	
	\begin{proof}
		The density $\mu$ is spherically symmetric, hence
		$$\int_{\bR^n} x_ix_jd\mu(x)=0,$$
		for $i\ne j$, and		
		$$\int_{\bR^n} x_i^2d\mu(x)=\frac{1}{n}\int_{\bR^n} |x|^2d\mu(x),$$
		for $i=1,2,...,n$.
		Integration in spherical coordinates and using Laplace asymptotic method yields
		\begin{align*}
			\int |x|^2d\mu(x)&=\frac{\displaystyle\int_0^{\sqrt{\frac{3}{7}n}} \frac{r^{n+1}}{\left(\sqrt{\frac{3}{7}n}\right)^{n-1}}dr+\displaystyle\int_{\sqrt{\frac{3}{7}n}}^{2\sqrt{\frac{3}{7}n}}r^2dr+
	\left(\frac{e^{\sqrt{2\sqrt{\frac{3}{7}n}}}}{2\sqrt{\frac{3}{7}n}}\right)^{n-1}\displaystyle\int_{2\sqrt{\frac{3}{7}n}}^{\infty}r^{n+1}e^{-(n-1)\sqrt{r}}dr}{\displaystyle\int_0^{\sqrt{\frac{3}{7}n}} \left(\frac{r}{\sqrt{\frac{3}{7}n}}\right)^{n-1}dr+\displaystyle\int_{\sqrt{\frac{3}{7}n}}^{2\sqrt{\frac{3}{7}n}}dr+
	\left(\frac{e^{\sqrt{2\sqrt{\frac{3}{7}n}}}}{2\sqrt{\frac{3}{7}n}}\right)^{n-1}\displaystyle\int_{2\sqrt{\frac{3}{7}n}}^{\infty}r^{n-1}e^{-(n-1)\sqrt{r}}dr}\\
	&=\frac{\sqrt{\frac{3}{7}}n^{\frac{3}{2}}+O\left(\sqrt{n}\right)}{\sqrt{\frac{3}{7}n}+O\left(\frac{1}{\sqrt{n}}\right)}=n+O(1).
		\end{align*}
	\end{proof}
	
	\begin{prop} \label{concentration}
		Let $X$ be a random vector in  $\bR^n$ distributed according to $\mu$ from Proposition \ref{iso0}. Then,
		$$\bP\left(\sqrt{\frac{3}{7}n}\leq |X| \leq 2\sqrt{\frac{3}{7}n}\right)\geq 1-\frac{C}{n}.$$
	\end{prop}
	
	\begin{proof}
		By the same arguments as in Proposition \ref{iso0}
		\begin{align*}\bP\left(\sqrt{\frac{3}{7}n}\leq |X| \leq 2\sqrt{\frac{3}{7}n}\right)&=\frac{\displaystyle\int_{\sqrt{\frac{3}{7}n}}^{2\sqrt{\frac{3}{7}n}}dr}
		{\sqrt{\frac{3}{7}n}+O\left(\frac{1}{\sqrt{n}}\right)}=1+O\left(\frac{1}{n}\right).
	\end{align*}
	\end{proof}
	
	\begin{prop} \label{iso} Let $X$ be a random vector in $\bR^n$ distributed according to $\mu$ from Proposition \ref{iso0}, and let $\widetilde{X}$ be a random variable distributed according to $d\widetilde{\mu}=\widetilde{C_n}\left(1-\frac{(n-1)f(|x|)}{N}\right)^N_+$.
	Then for $N\geq n^{\frac{5}{2}}\log^2 n$, $\widetilde{X}$ is essentially isotropic, namely
	$$\int x_ix_jd\widetilde{\mu}(x) = (1+\varepsilon_n')\delta_{ij}$$
	for all $i,j=1,2,...,n$, when $\left|\varepsilon_n'\right|\leq\frac{C}{\sqrt{n}}$. Also
	$$\forall t,\quad\left|\bP(|X|\leq t)-\bP(|\tilde{X}|\leq t)\right|\leq \frac{C}{\sqrt{n}}.$$
	\end{prop}
	
	\begin{proof}
		The random vector $\widetilde{X}$ is spherically symmetric. Hence
		$$\int_{\bR^n} x_ix_jd\widetilde{\mu}(x)=0,$$
		for $i\ne j$, and		
		$$\int_{\bR^n} x_i^2d\widetilde{\mu}(x)=\frac{1}{n}\int_{\bR^n} |x|^2d\widetilde{\mu}(x),$$
		for $i=1,2,...,n$.
		Since both densities are spherically symmetric, we need to estimate the one-dimensional integrals
		$$I_k=\int_0^{\infty}r^k\left(e^{-(n-1)f(r)}-\left(1-\frac{(n-1)f(r)}{N}\right)^N_+\right)dr$$
		for $k=n-1,n+1$. Define $\alpha$ by the equation $\left(\sqrt{\alpha}-\sqrt{2\sqrt{\frac{3}{7}n}}+\log\left(2\sqrt{\frac{3}{7}n}\right)\right)(n-1)=\frac{N}{2}$,
		That is, for any $r\leq \alpha$ we have $\frac{(n-1)f(r)}{N}\leq \frac{1}{2}$.
		By Taylor's Theorem, for any $r\leq \alpha$,
		$$\left|\log\left(1-\frac{(n-1)f(r)}{N}\right)^N_+-(-(n-1)f(r))\right|\leq C\frac{(n-1)^2}{N}f^2(r).$$
		Hence, for any $r\leq \alpha$
		\begin{align*}
		\left|e^{-(n-1)f(r)}-\left(1-\frac{(n-1)}{N}f(r)\right)^N_+\right|&=e^{-(n-1)f(r)}\left|1-\exp\left((n-1)f(r)-\log\left(1-\frac{(n-1)}{N}f(r)\right)^N_+\right)\right|\\
		&\leq C\frac{n^2}{N}e^{-(n-1)f(r)}f^2(r).
		\end{align*}
		Note that
		$$\left|\int_{\alpha}^{\infty}\left(e^{-(n-1)f(r)}-\left(1-\frac{(n-1)}{N}f(r)\right)^N_+\right)dr\right|\leq C\int_{\alpha}^{\infty}e^{-(n-1)f(r)}dr\leq Ce^{-n}.$$
		Combining the above inequalities, we obtain
		$$|I_k|\leq C_1\frac{n^2}{N}\int_0^{\alpha}r^ke^{-(n-1)f(r)}f^2(r)dr+C_2e^{-n}\leq C\frac{n^2}{N}\int_0^{\infty}r^ke^{-(n-1)f(r)}f^2(r)dr.$$
		Hence,
		$$|I_{n-1}|\leq C\frac{n^2}{N}\left(\sqrt{n}\log^2n + O\left(\frac{\log^2n}{\sqrt{n}}\right)\right)\leq C_1,$$
		$$|I_{n+1}|\leq C\frac{n^2}{N}\left(n^{\frac{3}{2}}\log^2n + O\left(\log^2n\sqrt{n}\right)\right)\leq C_2n.$$
		By the estimation on $I_{n-1}$, and the calculations in Proposition \ref{iso0} we obtain
		\begin{align*}
		\left|\displaystyle\int_0^{\infty}\left(1-\frac{(n-1)}{N}f(r)\right)^N_+dr-\sqrt{\frac{3}{7}n}\right|&\leq
		\left|\displaystyle\int_0^{\infty}\left(1-\frac{(n-1)}{N}f(r)\right)^N_+dr-\displaystyle\int_0^{\infty}e^{-(n-1)f(r)}dr\right|+O\left(\frac{1}{\sqrt{n}}\right)\\
		&=|I_{n-1}|+O\left(\frac{1}{\sqrt{n}}\right)\leq C_1.
		\end{align*}
		Hence,
		\begin{itemize}
			\item
			$\left(\displaystyle\int_0^{\infty}\left(1-\frac{(n-1)}{N}f(r)\right)^N_+dr\right)^{-1}=\sqrt{\frac{1}{\frac{3}{7}n}}\left(1+O\left(\frac{1}{\sqrt{n}}\right)\right)$;
			\item
			$\forall t,\quad\left|\bP(|X|\leq t)-\bP(|\tilde{X}|\leq t)\right|\leq \frac{C}{\sqrt{n}}$.
		\end{itemize}
		By the estimation of $I_{n+1}$ we obtain,
		$$\left|\bE X_i^2-\bE \tilde{X}_i^2 \right|=\frac{1}{n}\left|\bE |X|^2-\bE |\tilde{X}|^2 \right|\leq C\frac{1}{\sqrt{n}}\frac{1}{n}|I_{n+1}|\leq \frac{C}{\sqrt{n}}.$$
	\end{proof}
	
	\begin{remark}
		It is possible to take $a\approx \sqrt{\frac{3}{7}n}$ in the definition of $f$, such that $\widetilde{X}$ is isotropic.
	\end{remark}
	
	We use the following estimation in our proof of Theorem \ref{thm:weakconvex}.
	\begin{prop} \label{prop:bernstein}
		Let $Z_1,..,Z_n$ be independent standard Gaussian random variables, and let $0<\delta<\frac{1}{2}$. Then,
		$$\bP\left(\left|\sqrt{Z_1^2+...+Z_n^2}-\sqrt{n}\right|\leq n^\delta\right)\geq 1-Ce^{-cn^{2\delta}},$$
		where $c,C>0$ are constants.
	\end{prop}
	
	\begin{proof}
		Note that
		$$\left|Z_1^2+\cdots+Z_n^2-n\right|=\left|\sqrt{Z_1^2+\cdots+Z_n^2}-\sqrt{n}\right|\left|\sqrt{Z_1^2+\cdots+Z_n^2}+\sqrt{n}\right|\geq \left|\sqrt{Z_1^2+\cdots+Z_n^2}-\sqrt{n}\right|\sqrt{n}.$$
		Therefore it is enough to show that
		$$\bP\left(\left|Z_1^2+\cdots+Z_n^2-n\right|\leq n^{\delta+\frac{1}{2}}\right)\geq 1-Ce^{-cn^{2\delta}}.$$
		Note that for all $m\geq 1$ and for all $i=1,...,n$, we have
		$$\bE|Z_i^2-1|^m\leq \sum_{k=1}^m\binom{m}{k}\bE Z_i^{2k}\leq 2^m(2m)!!\leq 4^mm!$$
		where $(2m)!!=1\cdot 3\cdot 5\cdots (2m-1)$.
		Hence, by Bernstein's inequality \cite{B} we obtain
		$$\bP\left(\left|(Z_1^2-1)+\cdots+(Z_n^2-1)\right|> n^{\frac{1}{2}+\delta}\right)\leq Ce^{-cn^{2\delta}}.$$
	\end{proof}
	
	We are now ready to prove Theorem \ref{thm:weakconvex}.
	\begin{proof} By Proposition \ref{counter_log}, the function $(n-1)f(|x|)$ is $\frac{1}{2}-$convex. Proposition \ref{func_bodies} with $N=n^{\frac{5}{2}}\log^2 n$ yields a $\frac{1}{2}-$convex body $K$.
	Let $X$ be a random vector distributed uniformly in $K$. By the definition of $K$ the marginal of $X$ with respect to the first $n$ coordinates has density proportional to $\left(1-\frac{(n-1)f(|x|)}{N}\right)^N_+$. Denote this subspace by $E$. By Proposition \ref{iso}, $\textrm{Proj}_E X$ is essentially
	 isotropic.
	 Let $G$ be a standard Gaussian random variable.
	In order to show that $Y=\textrm{Proj}_E X$ has no approximately Gaussian marginals, we examine $\bP(|\theta_0\cdot Y| \leq t)$, for any $\theta_0\in S^{n-1}$.
	Using the symmetry of $Y$ and the rotation invariance of $\sigma_{n-1}$, we obtain,
	\begin{align*}\bP(|\theta_0\cdot Y| \leq t) = \bE 1_{[0,t]}(|\theta_0\cdot Y|)&=\int_{S^{n-1}}\bE 1_{[0,t]}(|\theta\cdot Y|) d\sigma_{n-1}(\theta)&\\
	&=\bE \int_{S^{n-1}} 1_{[0,t]}(\theta_1|Y|) d\sigma_{n-1}(\theta),&\end{align*}
	where $\theta=(\theta_1,...,\theta_n)$.
	Let $Z=(Z_1,...,Z_n)$, where $Z_i$ are independent standard Gaussian random variable.
	Since $Z$ is invariant under rotations, $\frac{Z}{|Z|}$ is distributed uniformly on $S^{n-1}$. Hence,
	\begin{align*}\bP(|\theta_0\cdot Y| \leq t) &=\bP\left(|Z_1||Y|\leq t\sqrt{Z_1^2+\cdots+Z_n^2}\right).&\end{align*}
	By the Proposition \ref{prop:bernstein}, $\bP(|\sqrt{Z_1^2+\cdots+Z_n^2}-\sqrt{n}|\leq n^{\frac{1}{100}})\geq 1-Ce^{-cn^{\frac{1}{50}}}$.
	Hence,
	\begin{align}\label{marginal}\bP(|\theta_0\cdot Y| \leq t) &= \bP\left(|Z_1||Y|\leq t\sqrt{n}\left(1+O\left(n^{-\frac{1}{2}+\frac{1}{100}}\right)\right)\right)+O\left(e^{-cn^{\frac{1}{50}}}\right).\end{align}
	By Propositions \ref{iso} and \ref{concentration}, there exists a random vector $Y'$ such that
	$$\forall t\; \left|\bP \left(|Y'|\leq t\right)-\bP\left(|Y|\leq t\right)\right|\leq \frac{C}{\sqrt{n}},\;\;\bP\left(\sqrt{\frac{3}{7}n}\leq |Y'| \leq 2\sqrt{\frac{3}{7}n}\,\right)\geq 1-\frac{C}{n},$$
	and $|Y'|$ has constant density function on $\left[\sqrt{\frac{3}{7}n},2\sqrt{\frac{3}{7}n}\right]$. By the triangle inequality, for $W$ distributed uniformly on $\left[\sqrt{\frac{3}{7}},2\sqrt{\frac{3}{7}}\right]$ and any $\sqrt{\frac{3}{7}}\leq\alpha\leq\beta\leq2\sqrt{\frac{3}{7}}$ we have
	$$|\bP(\sqrt{n}\alpha\leq |Y| \leq \sqrt{n}\beta)-\bP(\alpha\leq W \leq \beta)|\leq \frac{C}{\sqrt{n}}.$$
	Combining with (\ref{marginal}),
	$$\bP(|\theta_0\cdot Y| \leq t) = \bP\left(|G|W\leq t(1+O(n^{-\frac{1}{2}+\frac{1}{100}}))\right)+O\left(\frac{1}{\sqrt{n}}\right).$$
	We conclude that $|Y\cdot\theta_0|$ is very close to a distribution which is the product of a Gaussian with a uniform random variable, and the latter distribution is far from Gaussian.
	\end{proof}
	

\end{document}